\documentclass[a4paper,reqno,11pt]{amsart}
\usepackage{amsmath}
\usepackage{amssymb}
\usepackage{cases}
\allowdisplaybreaks[4]
\newtheorem{thm}{Theorem}[section]
\newtheorem{prop}[thm]{Proposition}

\newtheorem{lem}[thm]{Lemma}

\newtheorem{cor}[thm]{Corollary}

\theoremstyle{definition}
\newtheorem{definition}[thm]{Definition}
\newtheorem{example}[thm]{Example}
\newtheorem{assum}[thm]{Assumption}

\theoremstyle{remark}
\newtheorem{remark}[thm]{Remark}

\numberwithin{equation}{section}

\newcommand{\bQ}{\mathbb{Q}}
\newcommand{\bP}{\mathbb{P}}

\newcommand\OO{{\mathcal{O}}}
\newcommand{\rounddown}[1]{\lfloor{#1}\rfloor}
\newcommand{\roundup}[1]{\lceil{#1}\rceil}

\newcommand{\bC}{{\mathbb C}}

\newcommand\lcm{{\text{l.c.m.}}}

\begin{document}

\title{On birational geometry of minimal threefolds with numerically trivial canonical divisors}
\date{\today}
\author{Chen Jiang}
\address{Graduate School of Mathematical Sciences, the University of Tokyo,
3-8-1 Komaba, Meguro-ku, Tokyo 153-8914, Japan.}
\email{cjiang@ms.u-tokyo.ac.jp}

\thanks{The author is supported by Grant-in-Aid for JSPS Fellows (KAKENHI No. 25-6549) and Program for Leading Graduate  Schools, MEXT, Japan}

\begin{abstract}
For a minimal $3$-fold $X$ with $K_X\equiv 0$ and a nef and big Weil divisor $L$ on $X$, we investigate the birational geometry inspired by $L$. We prove that $|mL|$ and $|K_X+mL|$ give birational maps for all $m\geq 17$. The result remains true under weaker assumption that $L$ is big and has no stable base components.
\end{abstract} 
\maketitle
\pagestyle{myheadings} \markboth{\hfill  C. Jiang
\hfill}{\hfill On birational geometry of $3$-folds with $K\equiv 0$\hfill}

\tableofcontents

\section{Introduction}

A normal projective variety $X$ is said to be {\it minimal} if $X$ has at worst $\bQ$-factorial terminal singularities and the canonical divisor $K_X$ is nef.
According to Minimal Model Program, minimal varieties form a fundamental class in birational geometry.

Given an $n$-dimensional normal projective variety $X$ with mild singularities and a  big Weil divisor $L$ on $X$, we are interested in the geometry of the rational map $\Phi_{|mL|}$ defined by the linear system $|mL|$. By definition, $\Phi_{|mL|}$ is  birational onto its image when $m$ is sufficiently large.  Therefore it is interesting to find such a practical number $m(n)$, depending only on $\dim X$, which stably guarantees the birationality of $\Phi_{|mL|}$. In fact, the following three special cases are the most interesting:

{\bf Case 1.} $K_X$ is nef and big, $L=K_X$;

{\bf Case 2.} $K_X\equiv 0$, $L$ is an arbitrary nef and big Weil divisor;

{\bf Case 3.} $-K_X$ is nef and big, $L=-K_X$.

It is an interesting exercise to deal the case $X$ being a smooth curve or surface. We recall some known results on surfaces.
\begin{thm}[{cf. Bombieri \cite{Bomb}, Reider \cite{Reider}}]
Let $S$ be a smooth surface.
\begin{itemize}
\item[(1)] If $K_S$ is nef and big, then $|mK_S|$ gives a birational map for $m\geq 5$;
\item[(2)] If $K_S\equiv 0$, then $|mL|$ gives a birational map for $m\geq 3$ and $L$ an arbitrary nef and big divisor;
\item[(3)] If $-K_S$ is nef and big, then $|-mK_S|$ gives a birational map for $m\geq 3$.
\end{itemize}
\end{thm}

Smooth threefolds were studied by Matsuki \cite{Matsuki},  M. Chen \cite{Chen6}, Ando \cite{Ando}, Fukuda \cite{F},  Oguiso \cite{O}, and many others, and we have the following known results.
\begin{thm}[{Chen \cite{Chen6}, Fukuda \cite{F}}]
Let $X$ be a smooth $3$-fold.
\begin{itemize}
\item[(1)] If $K_X$ is nef and big, then $|mK_X|$ gives a birational map for $m\geq 6$;
\item[(2)] If $K_X\equiv 0$, then $|mL|$ gives a birational map for $m\geq 6$ and $L$ an arbitrary nef and big divisor;
\item[(3)] If $-K_X$ is nef and big, then $|-mK_X|$ gives a birational map for $m\geq 4$.
\end{itemize}
\end{thm}

When $X$ is a $3$-fold with $\bQ$-factorial terminal singularities, Case 1 was systematically treated by J. A. Chen  and M. Chen \cite{CC1}-\cite{CC3} and Case 3  is systematically treated by M. Chen and the author \cite{CJ}.
\begin{thm}[{Chen--Chen \cite{CC3}, Chen--Jiang \cite{CJ}}]
Let $X$ be a  $3$-fold  with $\bQ$-factorial terminal singularities.
\begin{itemize}
\item[(1)] If $K_X$ is nef and big, then $|mK_X|$ gives a birational map for $m\geq 61$;
\item[(2)] If $-K_X$ is nef and big, then $|-mK_X|$ gives a birational map for $m\geq 97$;
\item[(3)] If $-K_X$ is ample with $\rho(X)=1$, then $|-mK_X|$ gives a birational map for $m\geq 39$;
\end{itemize}
\end{thm}

%For a $3$-fold $X$,
%when $X$ is smooth, these cases are treated by Matsuki \cite{Matsuki}, Ando \cite{Ando}, Fukuda \cite{F}, Oguiso \cite{O} and many others;
%when $X$ is with terminal singularities, Case (1) is systematically treated by M. Chen and J. A. Chen \cite{CC1, CC2, CC3} and Case (3)  is systematically treated by M. Chen and the author \cite{CJ}.

Now we consider Case 2, when $X$ is a minimal $3$-fold with $K_X\equiv 0$ and $L$ is an arbitrary nef and big Weil divisor on $X$. If $X$ is smooth, then $|mL|$ gives a birational map for $m\geq 6$ by Fukuda \cite{F}. If $X$ is with Gorenstein terminal singularities and $q(X):=h^1(\OO_X)=0$, then $|mL|$ gives a birational map for $m\geq 5$ by Oguiso--Peternell \cite{OP}.  

The motivation of this paper is to study the birational geometry of minimal $3$-fold with $K\equiv 0$.
For an arbitrary nef and big Weil divisor $L$ on $X$, we investigate the birationality of the linear system $|mL|$ and the adjoint linear system  $|K_X+mL|$.
We remark that the behavior of these two linear systems are a litte diffenrent even though they are numerically equivalent. For example, if the local index $i(X)>1$, then $i(X)L$ is always a Cartier divisor while  $K_X+i(X)L$ can never be (see Section 2).

The difficulty arises from the singularities of $X$, and the assumption that $L$ is only a Weil divisor. If we assume that $L$ is Cartier, then the problem becomes relatively  easy and can be treated by the method of Fukuda \cite{F} using Reider's theorem \cite{Reider}. On the other hand, fortunately, the singularities of minimal $3$-folds with $K\equiv 0$ is not so complicated due to Kawamata \cite{Ka=0} and Morrison \cite{M=0}, and this makes it possible to deal with the birationality problem.

As the main result, we prove the following theorem. 
\begin{thm}\label{cy main} 
Let $X$ be a minimal $3$-fold with $K_X\equiv 0$ and  a nef and big Weil divisor  $L$. Then $|mL|$ and  $|K_X+mL|$ give birational maps for all $m\geq 17$.
\end{thm}

In fact, we prove a more general theorem.
\begin{thm}\label{cy main1} 
Let $X$ be a minimal $3$-fold with $K_X\equiv 0$, a nef and big Weil divisor L, and a Weil divisor $T\equiv 0$. Then $|K_X+mL+T|$ gives a birational map for all $m\geq 17$.
\end{thm}

Moreover, by Log Minimal Model Program, the assumption that $L$ is nef can be weaken. We say that a divisor $D$ {\it has no stable base components} if $|mD|$ has no base components for sufficiently divisible $m$.
\begin{thm}\label{cy main2} 
Let $X$ be a minimal $3$-fold with $K_X\equiv 0$,  a  big Weil divisor $L$ without stable base components,  and a Weil divisor $T\equiv 0$. Then $|K_X+mL+T|$ gives a birational map for all $m\geq 17$. In particular, $|mL|$ and  $|K_X+mL|$ give birational maps for all $m\geq 17$.
\end{thm}

As a by-product, we prove   a direct generalization of Fukuda \cite{F} and Oguiso--Peternell \cite{OP} which is optimal by the general weighted hypersurface $X_{10}\subset \mathbb{P}(1,1,1,2,5)$.
\begin{thm}[{=Theorem \ref{cy Gor}}]
Let $X$ be a minimal Gorenstein $3$-fold with $K_X\equiv 0$, a nef and big Weil divisor $L$, and a Weil divisor $T\equiv 0$.  Then $|K_X+mL+T|$ gives a birational map for all $m\geq 5$.
\end{thm}

For the convenience, we introduce the following definition.
\begin{definition}
 $(X, L, T)$ is called a {\it polarized triple} if $X$ is a minimal $3$-fold with $q(X)=0$ and $K_X\equiv 0$, $L$ is a nef and big Weil divisor, and $T$ is a numerically trivial Weil divisor on $X$.
\end{definition}

Note that we assume  $q(X)=0$ in the definition. The case that $q(X)>0$ is relatively easy and we treat it in Section \ref{cy Gor section} (see Theorem \ref{cy Gor}).

This paper is organized as follows. In Section \ref{cy preliminaries}, we recall some basic knowledge and facts. In Section  \ref{cy Gor section}, we treat Gorenstein case. We study the birationality  of polarized triples in  Section \ref{cy section birationality} and give an effective criterion for the birationality of linear systems. In the last section, to apply the birationality criterion, we estimate several quantities of polarized triples. We prove Theorems \ref{cy main1} and \ref{cy main2} in the last part.

{\bf Acknowledgment.} The author would like to express his gratitude to his supervisor Professor Yujiro Kawamata for  suggestions and encouragement.
He appreciates the very effective discussion with Professors Meng Chen and Keiji Oguiso during the preparation of this paper. Part of this paper was written during the author's visit to Fudan University and he would like to thank for the hospitality and support. The author would like to thank the anonymous reviewer for his valuable comments and suggestions to improve the explanation of the paper.
%\begin{remark}
%The method we use is valid for $X$ with canonical singularities. But the result will be weaker, since, for example, Proposition \ref{cy b zeta}(3) fails and one should replace it by a weaker estimation.
%\end{remark}

\section{Preliminaries}\label{cy preliminaries}

Throughout we work over an algebraically closed field $k$ of characteristic 0 (for instance, $k=\bC$). We adopt the standard notation in Koll\'ar--Mori \cite{KM}, and will freely use them.

Let $X$ be a minimal $3$-fold with $K_X\equiv 0$. Denote by $i(X)$ the {\it local index} of $X$, i.e. the Cartier index of $K_X$. By Kawamata \cite[Corollary 5.2]{K2}, for arbitrary Weil divisor $D$ on $X$, $i(X)D$ is a Cartier divisor.
By Kawamata \cite[Theorem 8.2]{K1}, $K_X\sim_{\bQ}0$ and we define
the {\it global index}
$$
I(X) = \min \{m \in \mathbb{N}\mid mK_X\sim 0\}. 
$$
Note that $i(X)| I(X)$.

For two linear systems $|A|$ and $|B|$, we write $|A|\preceq |B|$ if  there exists
an effective divisor $F$ such that $$|B|\supset |A|+F.$$ In particular, if $A\leq B$ as divisors, then $|A|\preceq |B|$.

\subsection{Rational map defined by a Weil divisor}\label{cy b setting}

Consider a $\bQ$-Cartier Weil divisor $D$ on $X$ with $h^0(X, D)\geq 2$. We study the rational map defined by $|D|$, say 
$$X\overset{\Phi_D}{\dashrightarrow} \bP^{h^0(D)-1}$$ which is
not necessarily well-defined everywhere. By Hironaka's big
theorem, we can take successive blow-ups $\pi: Y\rightarrow X$ such
that:
\begin{itemize}
\item [(1)] $Y$ is smooth projective;
\item [(2)] the movable part $|M|$ of the linear system
$|\rounddown{\pi^*(D)}|$ is base point free and, consequently,
the rational map $\gamma:=\Phi_D\circ \pi$ is a morphism;
\item [(3)] the support of the
union of $\pi_*^{-1}(D)$ and the exceptional divisors of $\pi$ is of
simple normal crossings.
\end{itemize}
Let $Y\overset{f}\longrightarrow \Gamma\overset{s}\longrightarrow Z$
be the Stein factorization of $\gamma$ with $Z:=\gamma(Y)\subset
\bP^{h^0(D)-1}$. We have the following commutative
diagram.\medskip

\begin{picture}(50,80) \put(100,0){$X$} \put(100,60){$Y$}
\put(170,0){$Z$} \put(170,60){$\Gamma$}
\put(112,65){\vector(1,0){53}} \put(106,55){\vector(0,-1){41}}
\put(175,55){\vector(0,-1){43}} \put(114,58){\vector(1,-1){49}}
\multiput(112,2.6)(5,0){11}{-} \put(162,5){\vector(1,0){4}}
\put(133,70){$f$} \put(180,30){$s$} \put(95,30){$\pi$}
\put(130,10){$\Phi_D$}\put(136,40){$\gamma$}
\end{picture}
\bigskip

{\bf Case $(f_{\rm{np}})$.} If $\dim(\Gamma)\geq 2$, a general
member $S$ of $|M|$ is a smooth projective surface by
Bertini's theorem. We say that $|D|$ {\it is not composed with
a pencil of surfaces}.

{\bf Case $(f_{\rm p})$.} If $\dim(\Gamma)=1$, i.e. $\dim\overline{\Phi_D(X)}=1$, a
general fiber $S$ of $f$ is an irreducible smooth projective surface
by Bertini's theorem. We may write
$$M=\sum_{i=1}^a S_i\equiv aS$$ where $S_i$ is a smooth fiber of $f$ for all $i$. We say that $|D|$ {\it is composed with
a pencil of surfaces}. It is clear that $a\geq h^0(D)-1$. Furthermore, $a=h^0(D)-1$  if  and only if $\Gamma\cong
\bP^1$, and   then we say that $|D|$ {\it is composed with
a rational pencil of surfaces}. In particular, if $q(X)=0$, then  $\Gamma\cong
\bP^1$ since $g(\Gamma)\leq q(Y)=q(X)=0$.

For another Weil divisor $D'$ satisfying $h^0(X,D')>1$,
we say that $|D|$ and $|D'|$ are {\it composed with the same pencil} if $|D|$ and $|D'|$ are composed with pencils and they define the same fibration structure $Y\rightarrow \Gamma$ on some smooth model $Y$. In particular, $|D|$ and $|D'|$ are {not composed with the same pencil} if one of them is not composed with  a pencil.
%We have the following lemma for rational pencils.
%\begin{lem}[{\cite[Lemma 2.2]{CJ}]}]\label{cy pencils}
%If $|D_1|$ and $|D_2|$ are composed with rational pencils of surfaces and $D_1\leq D_2$, then they are composed %with the same pencil of surfaces.
%\end{lem}

Define
$$\iota=\iota(D):=\begin{cases} 1, & \text{Case\ } (f_{\text{np}});\\
a, & \text{Case\ } (f_{\text{p}}).
\end{cases}$$
Clearly, in both cases, $M\equiv \iota S$ with $\iota\geq 1$.

\begin{definition} For both Case $(f_{\text{np}})$ and Case $(f_{\text{p}})$, we call $S$
{\it a general irreducible element of $|M|$}.
\end{definition}

%Note that we can also define that a base point free linear system on a surface is composed with a pencil of curves (which might be irrational), and define a generic irreducible element of a base point free linear system on a surface.
We may also define ``a general irreducible element'' of a moving linear system on any surface in the similar way. 

%Restricting our interest to special cases, we take  $D=m_0L_0$ at the very beginning assuming that $h^0(m_0L_0)\geq 2$ for some integer $m_0>0$ and $L_0\equiv L$. We would like to study the geometry of $X$ induced by $\Phi_D$. 

\subsection{Reid's Riemann--Roch formula}

A {\it basket} $B$  is a collection of pairs of integers (permitting
weights), say $\{(b_i,r_i)\mid i=1, \cdots, s; b_i\ \text{is coprime
 to}\ r_i\}$.  For simplicity, we will alternatively write a basket as follows, 
 say
$$B=\{(1,2), (1,2), (2,5)\}=\{2\times (1,2), (2,5)\}.$$

Let $X$ be a  3-fold with $\bQ$-factorial terminal singularities. According to Reid
\cite{YPG},  for a Weil divisor $D$ on $X$, 
$$
\chi(D)=\chi(\OO_X)+\frac{1}{12}D(D-K_X)(2D-K_X)+\frac{1}{12}(D\cdot c_2)+\sum_Qc_Q(D),
$$
where the last sum runs over Reid's basket of orbifold points. If the orbifold point $Q$ is of type $\frac{1}{r_Q}(1,-1,b_Q)$ and $i_Q=i_Q(D)$ is the local index of divisor $D$ at $Q$ (i.e. $D\sim i_QK_X$ around $Q$, $0\leq i_Q< r$),  then
$$
c_Q(D)=-\frac{i_Q(r_Q^2-1)}{12r_Q}+\sum_{j=0}^{i_Q-1}\frac{\overline{jb_Q}(r_Q-\overline{jb_Q})}{2r_Q}.
$$
Here the symbol $\overline{\cdot}$ means the smallest residue mod $r$ and $\sum_{j=0}^{-1}:=0$.   
We can write Reid's basket as $B_X=\{(b_Q, r_Q)\}_Q$. Note that we may assume $0<b_Q\leq \frac{r_Q}{2}$.  Recall that $i(X)=\lcm\{r_Q\in B_X\}$.

Let $X$ be a minimal $3$-fold with $K_X\equiv 0$. 
Note that for arbitrary nef and big Weil divisor  $H$, Kawamata--Viehweg
vanishing theorem \cite[Theorem 1-2-5]{KMM} implies
$$h^i(H)=h^i(K_X+(H-K_X))=0$$
for all $i>0$.
For a nef and big Weil divisor $L$ and a Weil divisor $T\equiv 0$, Reid's formula gives
$$
h^0(mL+T)=\chi(\OO_X)+\frac{m^3}{6}L^3+\frac{m}{12}(L\cdot c_2)+\sum_Qc_Q(mL+T).
$$
We make some remarks on estimating this formula. Recall that by Miyaoka \cite{Mi}, $c_2$ is pseudo-effective and hence $(L\cdot c_2)\geq 0$ holds.
Also Reid's formula or Kawamata \cite[Theorem 2.4]{Ka=0} gives
\begin{align}\label{cy chi}
\chi(\OO_X)=\sum_Q\frac{r_Q^2-1}{24r_Q}.
\end{align}
We define 
$$
\lambda(L):=\frac{1}{6}L^3+\frac{1}{12}(L\cdot c_2).
$$
Note that $\lambda(L)$ is a numerical invariant of $L$. We can rewrite Reid's formula as following:
$$
h^0(mL+T)=\chi(\OO_X)+\frac{m^3-m}{6}L^3+m\lambda(L)+\sum_Qc_Q(mL+T).
$$
We have the following lemma.
\begin{lem}\label{cy lambda}
$i(X)\lambda(L)\in \mathbb{Z}_{>0}$. In particular, $\lambda(L)\geq \frac{1}{i(X)}.$
\end{lem}
\begin{proof}
For a singular point $Q$ of type $(b,r)$, note that  if  $i$ runs over $\{0,1,\cdots, r-1\}$ then so does the local index of $L+iK_X$  at $Q$. Hence we have
\begin{align*}
{}&\sum_{i=0}^{r-1}c_Q(L+iK_X)\\
={}&\sum_{i=0}^{r-1}\Big(-\frac{i(r^2-1)}{12r}+\sum_{j=0}^{i-1}\frac{\overline{jb}(r-\overline{jb})}{2r}\Big)\\
={}&-\frac{(r-1)(r^2-1)}{24}+\sum_{i=1}^{r-1}\sum_{j=0}^{i-1}\frac{\overline{jb}(r-\overline{jb})}{2r}\\
={}&-\frac{(r-1)(r^2-1)}{24}+\sum_{j=0}^{r-2}\sum_{i=j+1}^{r-1}\frac{\overline{jb}(r-\overline{jb})}{2r}\\
={}&-\frac{(r-1)(r^2-1)}{24}+\sum_{j=1}^{r-2}(r-1-j)\frac{\overline{jb}(r-\overline{jb})}{2r}\\
={}&-\frac{(r-1)(r^2-1)}{24}+\sum_{k=2}^{r-1}(k-1)\frac{\overline{kb}(r-\overline{kb})}{2r} {}& (k=r-j)\\
={}&-\frac{(r-1)(r^2-1)}{24}+\frac{1}{2}\sum_{k=1}^{r-1}((r-1-k)+(k-1))\frac{\overline{kb}(r-\overline{kb})}{2r}\\
={}&-\frac{(r-1)(r^2-1)}{24}+\frac{r-2}{2}\sum_{k=0}^{r-1}\frac{\overline{kb}(r-\overline{kb})}{2r}\\
={}&-\frac{(r-1)(r^2-1)}{24}+\frac{r-2}{2}\sum_{j=0}^{r-1}\frac{j(r-j)}{2r}\\
={}&-\frac{r^2-1}{24}.
\end{align*}
Hence by Reid's formula,
\begin{align*}
{}&\sum_{i=0}^{i(X)-1}h^0(L+iK_X)\\
={}&\sum_{i=0}^{i(X)-1}\Big(\chi(\OO_X)+\lambda(L)+\sum_Qc_Q(L+iK_X)\Big)\\
={}&i(X)\chi(\OO_X)+i(X)\lambda(L)+\sum_Q\Big(-\frac{r^2_Q-1}{24}\cdot \frac{i(X)}{r_Q}\Big)\\
={}&i(X)\lambda(L).
\end{align*}
Hence $i(X)\lambda(L)\in \mathbb{Z}$. On the other hand, $\lambda(L) {>0}$ since $L$ is nef and big. 
\end{proof}

We have the following lemma for intersection numbers.
\begin{lem}\label{cy L3}
Let $X$ be a normal projective $3$-fold with $\bQ$-factorial terminal singularities. Recall that $i(X)$ is the local index of $X$, i.e. the Cartier index of $K_X$. Then for Weil divisors $D_1$, $D_2$, and $D_3$ on $X$, $(i(X)D_1\cdot D_2\cdot D_3)\in \mathbb{Z}$. In particular, if $L$ is a nef and big Weil divisor on $X$, then $L^3\geq \frac{1}{i(X)}$.
\end{lem}
\begin{proof}
Recall that by Kawamata  \cite[Corollary 5.2]{K2}, $i(X)D_1$ is Cartier.
Take a resolution of isolated singularities $\phi: W\rightarrow X$.
We may write $K_{W}=\phi^*(K_X)+\Delta$ where
$\Delta$ is an exceptional effective $\bQ$-divisor over those
isolated terminal singularities on $X$. Denote by $D_i'$ the strict transform of $D_i$ on $W$ for $i=1,2,3$. By intersection theory,
we have 
\begin{align*}
{}&(i(X)D_1\cdot D_2\cdot D_3)_X\\
={}&(\phi^*(i(X)D_1)\cdot\phi^*(D_2)\cdot D_3')_W\\
={}&(\phi^*(i(X)D_1)\cdot D_2'\cdot D_3')_W
\end{align*}
 is an
integer. 
\end{proof}

\subsection{Some facts about minimal $3$-folds with $K\equiv 0$}\label{cy facts}

We collect some facts about minimal $3$-folds with $K\equiv 0$ proved by Kawamata \cite{Ka=0} and Morrison \cite{M=0}.

\begin{thm}[{\cite{Ka=0,M=0}}]\label{cy fact}
Let $X$ be a minimal $3$-fold with $K_X\equiv 0$. The following facts hold:
\begin{itemize}
\item[(1)] $0\leq \chi(\OO_X)\leq 4$;
\item[(2)] $\chi(\OO_X)=0$ if and only if $X$ has Gorenstein singularities;
\item[(3)] If $q(X)>0$, then $X$ is smooth;
\item[(4)] If $q(X)=0$ and $\chi(\OO_X)\geq 2$, then 
$I(X)\in\{2,3,4,6\}$;
\item[(5)] If $q(X)=0$ and $\chi(\OO_X)=1$, then 
$$I(X)\in\{2,3,4,5,6,8,10,12\};$$
\item[(6)] If $I(X)\in\{ 5, 8,10,12\}$, then $\chi(\OO_X)=1$, $q(X)=h^2(\OO_X)=0$, $i(X)=I(X)$, and the singular points can be described explicitly by Morrison \cite[Proposition 3]{M=0}.
\end{itemize}
\end{thm}

\begin{proof}
(1) is proved by Kawamata \cite[Theorem 3.1]{Ka=0}. (2) is a direct consequence of equality (\ref{cy chi}). (3) is proved by Kawamata \cite{K1} and \cite{Ka=0} (see \cite[Section 1]{M=0}). (4) is proved by Morrison \cite[Proposition 1, Proof of Theorem 1]{M=0} and (5) is proved by Morrison \cite[Proposition 3, Proof of Theorem 2]{M=0}. (6) is a direct consequence of (2)-(5) and Morrison \cite[Proposition 3]{M=0}. 
\end{proof}

\section{Gorenstein case}\label{cy Gor section}

 Throughout this section, we assume that $X$ is a minimal Gorenstein $3$-fold with $K_X\equiv 0$ and  $L$ is a nef and big Weil divisor on $X$. Note that $L$ is a Cartier divisor  since $i(X)=1$. Recall that we have a {\it canonical model} $\mu:(X, L)\rightarrow (Z,H)$ such that $Z$ is a $3$-fold with canonical singularities and $\mu^*K_Z=K_X$, $H$ is an ample Catier divisor with $L=\mu^*H$ (cf. \cite[Lemma 0.2]{OP}).
\begin{lem}[{cf. \cite[Lemma 1.1]{OP}}]\label{cy 111}
Let $D$ be a divisor on $X$. Then
\begin{itemize}
\item[(1)] $(D\cdot L^2)^2\geq (D^2\cdot L)(L^3)$;
\item[(2)] $D\cdot L^2\equiv D^2\cdot L \mod 2;$
\item[(3)] If $D\cdot L^2=1$ and $D^2\cdot L\geq 0$, then $ D^2\cdot L=L^3=1$.
\end{itemize}
\end{lem}
\begin{proof}
See the proof of \cite[Lemma 1.1]{OP}. Note that $K_X\equiv 0$ is sufficient in the proof. 
\end{proof}

We prove Theorem \ref{cy main1} for the Gorenstein case. It is  a direct generalization of Fukuda \cite{F} and Oguiso--Peternell \cite{OP}, and we follow their ideas.
\begin{thm}\label{cy Gor}
Let $X$ be a minimal Gorenstein $3$-fold with $K_X\equiv 0$, a nef and big Weil divisor $L$, and a Weil divisor $T\equiv 0$.  Then $|K_X+mL+T|$ gives a birational map for all $m\geq 5$.
\end{thm}
\begin{proof}
Note that $L$ and $T$ are Cartier divisors since $i(X)=1$. 

{\bf Case 1.} $\dim \Phi_{|L|}(X)\geq 1.$

Take a resolution $\pi: Y\rightarrow X$. Consider the linear system $|K_Y+m\pi^*L+\pi^*T|$. Note that 
$$
\dim \Phi_{|\pi^*L|}(Y)=\dim \Phi_{|L|}(X)\geq 1.
$$
%for $m\geq 2$ by Riemann--Roch formula (see \cite[Lemma C]{F}). 
By \cite[Key Lemma]{F} with $R=\pi^*L$, $r_0=4$, and $r_1=1$, $|K_Y+m\pi^*L+\pi^*T|$ gives a birational map for all $m\geq 5$. So $|K_X+mL+T|$ gives a birational map for all $m\geq 5$.
\smallskip

{\bf Case 2.} $\dim \Phi_{|L|}(X)\leq 0.$

In this case, since $h^0(L)>0$ by Riemann--Roch formula, we have  $h^0(L)=1$. By Riemann--Roch formula again,
$$
h^0(2L)=\frac{1}{6}(2^3-2)L^3+2h^0(L)=L^3+2.
$$

First, we assume that $|2L|$ is composed with a pencil of surfaces. Set $D:=2L$ and keep the same notation as in Subsection \ref{cy b setting}.
Then we have 
$$2\pi^*(L)\geq M\equiv aS \geq (h^0(2L)-1)S=(L^3+1)S.$$ 
Thus we have 
$2L^3\geq
(L^3+1)(\pi^*(L)^2\cdot S)$. This implies that $L^2\cdot \pi_*S=\pi^*(L)^2\cdot S=1$ since $\pi^*(L)^2\cdot S>0$.
On the other hand, $L \cdot (\pi_*S)^2=\pi^*(L) \cdot \pi^*\pi_*S \cdot  S\geq 0$. Hence by Lemma \ref{cy 111}(3), $L \cdot (\pi_*S)^2=L^3=1$. Hence $M\equiv (L^3+1)S=2S$, in particular, $|2L|$ is composed with a rational pencil (see Case $(f_{\rm p})$ in subsection 2.1). Consider the canonical model $(Z, H)$. Since $h^0(H)=h^0(L)=1$ and $H^3=L^3=1$, there exists an irreducible surface $G$ such that $|H|=\{G\}$. Denote by $G'$ the strict transform of $G$. Then we may write $2L\sim 2G'+2E$ for some $\mu$-exceptional divisor $E$. Note that $\text{Mov}|2L|=|2\pi_*S|$, hence $2\pi_*S\sim 2G'+F$ for some effective $\mu$-exceptional divisor $F$. Note that $|2\pi_*S|$ is a rational pencil by construction, which means that, every element in $|2\pi_*S|$ can be written as the form $\pi_*S_1+\pi_*S_2$ with some $S_1\sim S_2\sim S$. Hence $\pi_*S_1+\pi_*S_2= 2G'+F$ (not only linear equivalence but equality). Hence $\pi_*S_1=G'+E'$ for some effective $\mu$-exceptional divisor $E'$ by the irreduciblity of $G'$. But this implies $\dim |\pi_* S|=0$, a contradiction.

Hence $|2L|$ is not composed with a pencil of surface. Set $D:=2L$ and keep the same notation as in Section \ref{cy b setting}.
Then we have $2\pi^*(L)= |M|+F$ such that $|M|$ is base point free. Consider a smooth element $N$ in $|M|$. Note that $|K_X+mL+T|$ gives a birational map if so does the restriction $|K_Y+N+\pi^*((m-2)L+T)||_N$ by Lemma \ref{cy b reduction} and birationality principle (cf. \cite[Lemma 1.3]{OP}). On the other hand, Kawamata--Viehweg vanishing theorem and adjunction formula give $$|K_Y+N+\pi^*((m-2)L+T)||_N=|K_N+\pi^*((m-2)L+T)|_N|.$$ 
Reider's theorem (cf. \cite{Reider}) implies that $|K_N+\pi^*((m-2)L+T)|_N|$ gives a birational map for $m\geq 5$ if $(\pi^*L)^2\cdot N\geq 2$. Now we assume that $(\pi^*L)^2\cdot N=1$, then Lemma \ref{cy 111}(3) implies that $L^3=L^2\cdot \pi_*N=L\cdot (\pi_*N)^2=1$. Consider the canonical model $(Z, H)$. Since $h^0(H)=H^3=1$, and $L^2\cdot \pi_*N=1$, a similar argument implies 
 $\dim |\pi_*N|=0$, a contradiction.

We completed the proof. 
\end{proof}

By Theorems \ref{cy Gor} and   \ref{cy fact}(2)(3), to prove Theorem \ref{cy main1}, we only need to consider polarized triples $(X,L,T)$ with $\chi(\OO_X)>0$.

\section{Birationality criterion}\label{cy section birationality}

In this section, we give a criterion for the birationality of polarized triples. The methods using in this section are mainly developed in Chen--Chen \cite{CC2} and Chen \cite{C}, and latter modified in Chen--Jiang \cite{CJ}.

\subsection{Main reduction}
 
Firstly, we reduce the birationality problem on $X$ to that on its smooth model $Y$.

\begin{lem}[{cf. \cite[Lemma 2.5]{C}}]\label{cy b Hn} Let $W$ be a normal projective variety on which there is
an integral Weil $\bQ$-Cartier divisor $D$. Let $h: V\longrightarrow
W$ be any resolution of singularities. Assume that $E$ is an
effective exceptional $\bQ$-divisor on $V$ with $h^*(D)+E$ a Cartier
divisor on $V$. Then
$$h_*\OO_V(h^*(D)+E)=\OO_W(D)$$
where $\OO_W(D)$ is the reflexive sheaf corresponding to the Weil
divisor $D$.
\end{lem}

\begin{lem}\label{cy b reduction}
Let $X$ be a normal projective variety with $\bQ$-factorial terminal singularities, $D$ be a Weil divisor on $X$ and 
$\pi:Y\longrightarrow X$ be a resolution. Then 
$\Phi_{|K_X+D|}$ is birational if and only if so is
$\Phi_{|K_Y+\roundup{\pi^*(D)}|}$.
\end{lem}

\begin{proof} Recall that
$$K_{Y}=\pi^*(K_X)+E_{\pi}$$ 
where $E_{\pi}$ is an effective $\bQ$-Cartier $\bQ$-divisor since $X$ has at worst terminal singularities.
We have
\begin{align*}
{}& K_Y+\roundup{\pi^*(D)}\\
={}& \pi^*(K_X)+ E_{\pi}+\pi^*(D)+E\\
={}& \pi^*(K_X+D)+ E_{\pi}+E
\end{align*}
where $E_{\pi}+E$ is an effective $\bQ$-divisor on $Y$ exceptional over $X$. Lemma \ref{cy b Hn}
implies $$\pi_*\OO_Y(K_Y+\roundup{\pi^*(D)})=\OO_X(K_X+D).$$
Hence $\Phi_{|K_X+D|}$ is birational if and only if so is
$\Phi_{|K_Y+\roundup{\pi^*(D)}|}$. 
\end{proof}

%In practice, we will consider $L$ be a nef and big Weil divisor on $X$, $R$ be a numerically trivial Weil divisor on $X$. 
%Noting that
%\begin{align*}
%H^0(\OO_X(K_X+D))\cong{}& H^0(\OO_Y(\rounddown{\pi^*(K_X+D)}))\\
%\cong {}&H^0(\OO_Y(K_Y+\roundup{\pi^*(D)})), 
%\end{align*}
%we denote by $|M_{m}|$ the movable part of
%$|\rounddown{\pi^*(K_X+mL+R)}|$. We have the equality:
%\begin{align}\label{cy b 2.1}
%\rounddown{\pi^*(K_X+mL+R)}=M_{m}+F_{m} 
%\end{align}
%where $F_m$ is an effective $\bQ$-divisor. 
%Another direct consequence is that we may write:
%$$K_Y+\roundup{\pi^*(mL+R)}\sim M_{m}+N_{m}$$ where $N_{m}$is the fixed part.

\subsection{Key theorem}\label{cy b keythm}
Let $(X, L, T)$ be a polarized triple. Take a Weil divisor $L_0$ such that $L_0\equiv L$. Suppose that $h^0(m_0L_0)\geq
2$ for some integer $m_0>0$.  Suppose that $m_1\geq m_0$ is an integer with $h^0(m_1L_0)\geq
2$ and that $|m_1L_0|$ and $|m_0L_0|$ are not composed with the same pencil. Note that in  application, if $|m_0L_0|$ is not composed with a pencil, we can just take $m_1=m_0$.
%Recall that, for any divisor $D$ with $h^0(D)>1$, 
%$$\iota(D)=\begin{cases} 1, & |-mK_X| \text{ is not composed with a pencil}; \\
%P_{-m}-1, &|-mK_X| \text{ is composed with a pencil} .
%\end{cases}$$

Set $D:=m_0L_0$ and keep the same notation as in Subsection \ref{cy b setting}. We may modify the resolution $\pi$ in Subsection \ref{cy b setting} such
that the movable part $|M_{m}|$ of $|\rounddown{\pi^*(mL_0)}|$ is base point free for all  $m_0\leq m\leq m_1$. 
Set $\iota_{m}:=\iota(mL_0)$ defined in Subsection \ref{cy b setting}. Recall that, for any integer $m$ with $h^0(mL_0)>1$, 
$$\iota_m=\begin{cases} 1, &\text{if } |mL_0| \text{ is not composed with a pencil}; \\
h^0(mL_0)-1, &\text{if }|mL_0| \text{ is composed with a pencil}.
\end{cases}$$

Pick a general irreducible element $S$ of $|M_{m_0}|$.  
We have 
$$m_0\pi^*(L_0)=\iota_{m_0}S+F_{m_0}$$ for some effective
$\bQ$-divisor $F_{m_0}$. 
 In particular, we see that 
$$\frac{m_0}{\iota_{m_0}}\pi^*(L_0)-S\sim_{\bQ}  \text{effective}\ \bQ\text{-divisor}.$$
Define the real number 
$$\mu_0=\mu_0(|S|):=\text{inf}\{t\in \bQ^+ \mid t\pi^*(L_0)-S\sim_{\bQ} \text{effective}\ \bQ\text{-divisor}\}.$$

\begin{remark}\label{cy upper mu0}Clearly, we have $0<\mu_0\leq \frac{m_0}{\iota_{m_0}}\leq m_0$. 
%On the other hand, for any given ample divisor $H$ on $Y$, $\mu_0\geq \frac{(H^2\cdot S)}{(H^2\cdot \pi^*(L_0))}>0$. 
For all $k$ such that $|kL_0|$ and $|m_0L_0|$ are composed with the same pencil, we have 
$$k\pi^*(L_0)=\iota_kS+F_{k}$$ for some effective
$\bQ$-divisor $F_{k}$, and hence $\mu_0\leq  \frac{k}{\iota_k}$.
\end{remark}

By the assumption on $|m_1L_0|$, we know that $|G|=|{M_{m_1}}|_{S}|$ is a base point free linear system on $S$ and $h^0(S, G)\geq 2$. Denote by
$C$ a general irreducible element of $|G|$. Note that since $K_X\equiv 0$, $K_Y$ is pseudo-effective and hence $g(C)\geq 1$. Since $m_1\pi^*(L_0)\geq M_{m_1}$, we have
$$m_1\pi^*(L_0)|_S\equiv C+H$$ 
where $H$ is an effective
$\bQ$-divisor on $S$. 

We define two numbers which will be the key
invariants accounting for the birationality of  $\Phi_{|K_X+mL+T|}$.
They are
\begin{align*}
\zeta{}&:=(\pi^*(L)\cdot C)_Y=(\pi^*(L_0)\cdot C)_Y=(\pi^*(L_0)|_S\cdot C)_S\ \text{and}\\
\epsilon(m){}&:=(m-\mu_0-m_1)\zeta.
\end{align*}
Note that $\zeta$ and $\epsilon(m)$ are birational invariants by projection formula. 
Hence we can modify $\pi$ if necessary. Also note that $\zeta>0$ since $L$ is nef and big and $C$ is free. 
%where $\iota$ is defined in Subsection \ref{cy b setting} so that $M_{-m_0}\equiv \iota S$.

While studying the birationality of $\Phi_{|K_X+mL+T|}$, we will always check that the linear system
$\Lambda_m:=|K_Y+\roundup{\pi^*(mL+T)}|$ satisfies the
following assumption for some integer $m>0$.

\begin{assum}\label{cy b asum}Keep the notation as above.
\begin{itemize}
\item [(1)] The linear system $\Lambda_m$
distinguishes different general irreducible elements of $|M_{m_0}|$
(namely, $\Phi_{\Lambda_m}(S')\neq \Phi_{\Lambda_m}(S'')$ for two
different general irreducible elements $S'$, $S''$ of $|M_{m_0}|$).

\item [(2)] The linear system ${\Lambda_m}|_{S}$ distinguishes different general irreducible elements of the linear system $|G|=|{M_{m_1}}|_{S}|$ on
$S$.
\end{itemize}
\end{assum}

The following is the key theorem in this section.

\begin{thm}\label{cy b mb}Let $(X, L, T)$ be a polarized triple. Keep the notation as above. Let $m>0$ be an integer. If
Assumption \ref{cy b asum} is satisfied and $\epsilon(m)>2$, then
$\Phi_{|K_X+mL+T|}$ is birational onto its image.
\end{thm}
\begin{proof} By Lemma \ref{cy b reduction}, we only need to prove the birationality of ${\Phi_{\Lambda_m}}$. Since Assumption \ref{cy b asum}(1) is satisfied, the usual birationality principle
reduces the birationality of ${\Phi_{\Lambda_m}}$ to that of
${\Phi_{\Lambda_m}}|_S$ for a general irreducible element $S$ of $|M_{m_0}|$.
Similarly, due to Assumption \ref{cy b asum}(2), we only need to prove
the birationality of ${\Phi_{\Lambda_m}}|_C$ for a general
irreducible element $C$ of $|G|$. Now we show how to restrict the linear system $\Lambda_m$ to $C$.

%Equality (\ref{cy b 2.1}) gives
%$$k_0\pi^*(-K_X)=M_{-k_0}+F_{k_0}$$ for some effective
%$\bQ$-divisor $F_{k_0}$. 

Now assume $\epsilon(m)>0$. 
 We can find a sufficiently large integer $n$ so that there exists a number $\mu_n\in \bQ^+$ with $0\leq \mu_n-\mu_0\leq \frac{1}{n}$, $\roundup{\epsilon(m,n)}=\roundup{\epsilon(m)}$ 
where $\epsilon(m,n):=(m-\mu_n-m_1)\zeta$ 
and 
$$\mu_n\pi^*(L_0)\sim_{\bQ} S+E_{n}$$ for an effective $\bQ$-divisor $E_{n}$. In particular, $\epsilon(m,n)>0$, and $\epsilon(m,n)>2$ if $\epsilon(m)>2$. Re-modify the resolution $\pi$ in Subsection \ref{cy b setting} so that $E_n$ has simple normal crossing support. 

For the given integer $m>0$, we have
\begin{align}\label{cy b subsys1}
|K_Y+\roundup{\pi^*(mL+T)-E_n}|\preceq
|K_Y+\roundup{\pi^*(mL+T)}|.
\end{align}
Since
$\epsilon(m, n)>0$, the $\bQ$-divisor
$$\pi^*(mL+T)-E_n-S\equiv (m-\mu_n)\pi^*(L)$$
is nef and big and thus $$H^1(Y,
K_Y+\roundup{\pi^*(mL+T)-E_n}-S)=0$$ by
Kawamata--Viehweg vanishing theorem. Hence we have
surjective map
\begin{align}\label{cy b surj1}
H^0(Y,K_Y+\roundup{\pi^*(mL+T)-E_n})\longrightarrow
H^0(S, K_S+L_{m,n}) 
\end{align} 
where
\begin{align}\label{cy b subsys2}
L_{m,n}:=(\roundup{\pi^*(mL+T)-E_n}-S)|_S\geq
\roundup{\mathcal{L}_{m,n}}
\end{align}
and ${\mathcal L}_{m,n}:=(\pi^*(mL+T)-E_n-S)|_S.$ Moreover, we have
$$m_1\pi^*(L_0)|_S\equiv C+H$$
for an effective $\bQ$-divisor $H$ on $S$ by the setting. Thus the $\bQ$-divisor
$${\mathcal
L}_{m,n}-H-C\equiv (m-\mu_n-m_1) \pi^*(L)|_S$$
is nef and big by $\epsilon(m,n)>0$. By Kawamata--Viehweg vanishing theorem again, $$H^1(S,
K_S+\roundup{{\mathcal
L}_{m,n}-H}-C)=0.$$ Therefore, we have surjective map
\begin{align}\label{cy b surj2}
H^0(S, K_S+\roundup{{\mathcal
L}_{m,n}-H})\longrightarrow
H^0(C, K_C+D_{m,n})
\end{align} 
where
\begin{align}\label{cy b subsys3}
D_{m,n}:=\roundup{{\mathcal
L}_{m,n}-H-C}|_C\geq
\roundup{\mathcal{D}_{m,n}}
\end{align} 
and
$\mathcal{D}_{m,n}:=({\mathcal
L}_{m,n}-H-C)|_C$ with
$\deg{\mathcal{D}_{m,n}}={\epsilon(m,n)}$.

Now by inequalities (\ref{cy b subsys1}), (\ref{cy b subsys2}), (\ref{cy b subsys3}), and surjective maps (\ref{cy b surj1}), (\ref{cy b surj2}), to prove
the birationality of ${\Phi_{\Lambda_m}}|_C$, it is
sufficient to prove that $|K_C+\roundup{\mathcal{D}_{m,n}}|$ gives a
birational map. Clearly this is the case whenever $\epsilon(m)>2$,
which in fact implies $\deg(\roundup{\mathcal{D}_{m,n}})\geq \roundup{\epsilon(m,n)}\geq 3$ and $K_C+\roundup{\mathcal{D}_{m,n}}$ is very ample. We complete the proof. 
\end{proof}

\begin{cor}\label{cy m00}  Keep the same notation as above. For any integer $m>0$, set
$$\epsilon(m,0):=(m-\frac{m_0}{\iota_{m_0}}-m_1)\zeta.$$ If  $\epsilon(m,0)>0$, then
$$\Lambda_m|_S\succeq |K_S+L_m|$$
where  $L_m:=(\roundup{\pi^*(mL+T)-\frac{1}{\iota_{m_0}}F_{m_0}}-S)|_S$.
\end{cor} 
\begin{proof} Recall that 
$$m_0\pi^*(L_0)=\iota_{m_0}S+F_{m_0}.$$
{}First of all, %relation (\ref{cy b subsys1}) reads 
%\begin{align}\label{cy b0 subsys1}
$$|K_Y+\roundup{\pi^*(mL+T)-\frac{1}{\iota_{m_0}}F_{m_0}}|\preceq
|K_Y+\roundup{\pi^*(mL+T)}|.
$$%\end{align}
In fact, as long as  $\epsilon(m,0)>0$, the proof of Theorem \ref{cy b mb} is valid except for the last paragraph. In explicit, sujective map (\ref{cy b surj1}) reads the following surjective map
%\begin{align}\label{cy b0 surj1}
$$H^0(Y,K_Y+\roundup{\pi^*(mL+T)-\frac{1}{\iota_{m_0}}F_{m_0}})\longrightarrow
H^0(S, K_S+L_m) 
$$%\end{align} 
where
%\begin{align}\label{cy b0 subsys2}
$$L_m=(\roundup{\pi^*(mL+T)-\frac{1}{\iota_{m_0}}F_{m_0}}-S)|_S.
$$%\end{align}
Hence we have proved the statement. 
\end{proof}

\subsection{Criterion}\label{cy b section birationality}

In order to apply Theorem \ref{cy b mb}, we need to verify Assumption \ref{cy b asum} and $\epsilon(m)>2$ in advance, for which one of the crucial steps is to estimate the lower bound of $\zeta$.

\begin{lem}\label{lemma zeta}
\begin{itemize}
\item[(1)] If $\epsilon(m)>1$, then $\zeta\geq \frac{2g(C)-2+\roundup{\epsilon(m)}}{m}$;

\item[(2)] Moreover, $$\zeta\geq \frac{2g(C)-1}{\mu_0+m_1+1};$$

\item[(3)] If $g(C)=1$, then  $\zeta\geq 1$;

\item[(4)] $i(X)\zeta\in \mathbb{Z}_{>0}$.
\end{itemize}
\end{lem}
\begin{proof}
(1). Recall that since $K_X\equiv 0$, $K_Y$ is pseudo-effective and hence $g(C)\geq 1$. In the proof of Theorem \ref{cy b mb}, if $\epsilon(m)>1$ then
$|K_C+\roundup{\mathcal{D}_{m,n}}|$ is base point free with
$$\deg(K_C+\roundup{\mathcal{D}_{m,n}})\geq 2g(C)-2+\roundup{\epsilon(m,n)}= 2g(C)-2+\roundup{\epsilon(m)}.$$
Denote by $\mathcal{N}_m$ the movable part of
$|K_S+\roundup{\mathcal{L}_{m,n}-H}|$. Note that 
\begin{align*}
{}& H^0(\OO_Y(\rounddown{\pi^*(K_X+mL+T)}))\\
\cong{}& H^0(\OO_X(K_X+mL+T))\\
\cong {}&H^0(\OO_Y(K_Y+\roundup{\pi^*(mL+T)})).
\end{align*}
Denote by $\mathcal{M}_m$ the movable part of $|\rounddown{\pi^*(K_X+mL+T)}|$.
Noting the relations
(\ref{cy b subsys1}-\ref{cy b subsys3}) while applying \cite[Lemma 2.7]{Camb}, we
get
$$\pi^*(K_X+mL+T)|_S\geq {\mathcal{M}_{m}}|_S\geq \mathcal{N}_m$$
and ${\mathcal{N}_m}|_C\geq K_C+\roundup{\mathcal{D}_{m,n}}$ since
the latter one is base point free.
So we have 
$$m\zeta=\pi^*(K_X+mL+T)|_S\cdot C\geq \mathcal{N}_m\cdot C\geq \deg(K_C+\roundup{\mathcal{D}_{m,n}}).$$
Hence 
$$m\zeta\geq
2g(C)-2+\roundup{\epsilon(m)}.$$

(2). Take $m'=\min\{m\mid\epsilon(m)>1\}$, then (1) implies
$\zeta\geq \frac{2g(C)}{m'}$. We may assume that $m'> \mu_0+m_1+1$ otherwise $\zeta\geq \frac{2g(C)}{\mu_0+m_1+1}$. Hence 
\begin{align*}
\epsilon(m'-1){}&=(m'-1-\mu_0-m_1)\zeta\\
{}&\geq (m'-1-\mu_0-m_1)\frac{2g(C)}{m'}.
\end{align*}
By the minimality of $m'$, it follows that $\epsilon(m'-1)\leq 1$. Hence $m'\leq \frac{2g(C)}{2g(C)-1}(\mu_0+m_1+1)$. Then $$\zeta\geq \frac{2g(C)}{m'}\geq \frac{2g(C)-1}{\mu_0+m_1+1}.$$

(3). Recall that 
$$
K_Y= \pi^* K_X+E_{\pi}\equiv E_{\pi},
$$
where $E_{\pi}$ is an effective $\bQ$-Cartier $\bQ$-divisor whose support contains all $\pi$-exceptional divisors since $X$ has at worst terminal singularities.
If $g(C)=1$, then
\begin{align*}
0{}&=((K_S+C)\cdot C)_S\\
{}&=(K_Y\cdot C)_Y+(S\cdot C)_Y+(C^2)_S\\
{}&=(E_{\pi}\cdot C)_Y+(S\cdot C)_Y+(C^2)_S.
\end{align*}
Since $C$ is free on a free surface $S$,  $(C^2)_S$, $(S\cdot C)_Y$, and $(E_{\pi}\cdot C)_Y$ are non-negative. Hence $(E_{\pi}\cdot C)_Y=0$, which implies that $(E\cdot C)_Y=0$ for any $\pi$-exceptional divisor $E$ on $Y$ since $X$ has at worst terminal singularities.
Hence $\zeta=(\pi^*L\cdot C)_Y$ is an integer. On the other hand, $\zeta>0$. Hence $\zeta\geq 1$.

(4). It follows from the fact that $i(X)L$ is Cartier. 
\end{proof}

\begin{prop}\label{cy b zeta}Let $m>0$ be an integer. Keep the same notation as in Subsection \ref{cy b keythm}. 
Then
$$
\zeta\geq\Big\lceil{i(X)\min\Big\{1,\frac{3}{\mu_0+m_1+1} \Big\}}\Big\rceil/i(X).
$$
\end{prop}

\begin{proof}
If $g(C)=1$, by Lemma \ref{lemma zeta}(3), $\zeta\geq 1$; if $g(C)\geq 2$, by Lemma \ref{lemma zeta}(2), $\zeta\geq \frac{3}{\mu_0+m_1+1}$.
Then by Lemma \ref{lemma zeta}(4),
$$
i(X)\zeta\geq\Big\lceil{i(X)\min\Big\{1,\frac{3}{\mu_0+m_1+1} \Big\}}\Big\rceil.
$$
We complete the proof. 
\end{proof}

Define 
$$
\rho_0:=\min\{k\in \mathbb{Z}_{>0}\mid h^{0}(mL+T')>0 \text{ for all } m\geq k \text{ and for all } T'\equiv 0 \},
$$
where $T'$ is assumed to be a  Weil divisor.

To verify Assumption \ref{cy b asum}(1), we have the following proposition.

\begin{prop}\label{cy b a1} Let $(X, L, T)$ be a polarized triple. Keep the same notation as Subsection \ref{cy b keythm}. Then
Assumption \ref{cy b asum}(1) is satisfied for all $m\geq m_0+\rho_0$.
\end{prop}
\begin{proof}We have
\begin{align*} 
{}& K_Y+\roundup{\pi^*(mL+T)}\\
\geq {}& K_Y+\roundup{\pi^*(mL+T-m_0L_0)+M_{m_0}}\\
={}& (K_Y+\roundup{\pi^*(mL+T-m_0L_0)})+M_{m_0}\\
\geq{}& M_{m_0}.
\end{align*} 
The last inequality is due to
\begin{align*}
{}&h^{0}(K_Y+\roundup{\pi^*(mL+T-m_0L_0)})\\
={}&h^{0}(K_X+mL+T-m_0L_0)>0
\end{align*}
when $m\geq m_0+\rho_0$ by Lemma \ref{cy b Hn} and the definition of $\rho_0$.

When $f:Y\rightarrow \Gamma$ is of type $(f_{\text{np}})$,
\cite[Lemma 2]{T} implies that $\Lambda_m$ can distinguish different
general irreducible elements of $|M_{m_0}|$. When $f$ is of type
$(f_{\text{p}})$, since the rational
pencil $|M_{m_0}|$ (recall that $q(X)=0$) can already separate different fibers of $f$,
$\Lambda_m$ can automatically distinguish different general irreducible
elements of $|M_{m_0}|$. 
\end{proof}
%\begin{lem}\label{cy b >6}
%Let $X$ be a weak $\bQ$-Fano
%3-fold. Assume that $P_{-m_0}>0$ for some positive integer $m_0$. Then $P_{-m}>0$ for all $m\geq 3m_0$. 
%\end{lem}
%\begin{proof}
%If $m_0=1$, it is trivial. If $m_0\geq 2$, then it follows from $P_{-m}>0$ for all $m\geq 6$ by \cite[Appendix]{C}.
%\end{proof}

It is slightly more complicated to verify Assumption \ref{cy b asum}(2).

\begin{lem}[{cf. \cite[Lemma 3.7]{C}}]\label{cy b S2} Let $R$ be a smooth projective surface with a base point free linear system $|G|$. Let $Q$ be
an arbitrary $\bQ$-divisor on $R$. % such that the fractional part $\{Q\}$ of $Q$ is with simple normal crossing support.
Denote by $C$ a general irreducible element of $|G|$. Then the linear system
$|K_R+\roundup{Q}+G|$ can distinguish different general irreducible
elements of $|G|$ under one of the following conditions:
\begin{itemize}
\item[(1)] $|G|$ is not
composed with an irrational pencil of curves and $K_R+\roundup{Q}$ is effective;

\item[(2)]$|G|$ is composed with an
irrational pencil of curves, $g(C)> 0$, and $Q$ is nef and big;
\end{itemize}
\end{lem}
\begin{proof} The statement corresponding to (1) follows from \cite[Lemma 2]{T} and
the fact that a rational pencil can automatically distinguish its
different general irreducible elements.

For situation (2), we pick a general irreducible element
$C$ of $|G|$. Then, since $h^0(R, G)\geq 2$, $G\equiv sC$ for some
integer $s\geq 2$ and $C^2=0$. Denote by $C_1$ and $C_2$ two general 
irreducible elements of $|G|$ such that $C_1+C_2\leq |G|$. Then Kawamata--Viehweg vanishing theorem
gives the surjective map
$$
H^0(R,K_R+\roundup{Q}+G)\longrightarrow H^0(C,K_{C_1}+D_1)\oplus
H^0(C_2,K_{C_2}+D_2)
$$ 
where $D_i:=(\roundup{Q}+G-C_i)|_{C_i}$ with $\deg(D_i)\geq Q\cdot C_i>0$ for $i=1,2$.

If $g(C)> 0$, Riemann--Roch formula gives
$h^0(C_i, K_{C_i}+D_i)>0$ for $i=1,2$. Thus $|K_R+\roundup{Q}+G|$ can
distinguish $C_1$ and $C_2$. 
\end{proof}

\begin{prop}\label{cy b a2} Let $(X, L, T)$ be a polarized triple. Keep the same notation as in Subsection
\ref{cy b keythm}. 
Then Assumption \ref{cy b asum}(2) is satisfied for all  $m\geq m_0+m_1+\rho_0$.
\end{prop}
\begin{proof}  Assuming $m\geq m_0+m_1+1$, 
we have $\epsilon(m,0)>0$,  and Corollary \ref{cy m00} implies that  
$${\Lambda_m}|_S\succeq |K_S+L_{m}|.$$ 
It suffices to prove that $|K_S+L_{m}|$ can distinguish different general irreducible elements of $|G|$.

For a suitable integer $m>0$, we have 
\begin{align*} 
{}&K_S+L_m\\
={}& K_Y|_S+\roundup{\pi^*(mL+T)-\frac{1}{\iota_{m_0}}F_{m_0}}|_S\\
\geq{}& (K_Y+\roundup{\pi^*(mL+T-(m_0+m_1)L_0)})|_S+{M_{m_1}}|_S.
\end{align*} 
Thus, if $|G|$ is not composed with an irrational pencil of curves,
$|K_S+L_m|$ can distinguish different irreducible elements provided that $$K_Y+\roundup{\pi^*(mL+T-(m_0+m_1)L_0)}$$ is effective, which holds for $m-m_0-m_1\geq \rho_0$.

Assume $|G|$ is composed with an irrational pencil of curves, we
have
\begin{align*}
{}&K_S+L_m\\
 \geq{}& K_S+\roundup{(\pi^*(mL+T)-\frac{1}{\iota_{m_0}}F_{m_0}-S)|_S}\\
\geq{}& K_S+\roundup{(\pi^*(mL+T-m_1L_0)-\frac{1}{\iota_{m_0}}F_{m_0}-S)|_S}+{M_{m_1}}|_S.
\end{align*} 
We can take $Q=(\pi^*(mL+T-m_1L_0)-\frac{1}{\iota_{m_0}}F_{m_0}-S)|_S$ in Lemma \ref{cy b S2} since $\epsilon(m,0)>0$.
Since $g(C)> 0$, Lemma \ref{cy b S2}(2) implies that Assumption \ref{cy b asum}(2) is satisfied for $m\geq m_0+m_1+1$. 

We complete the proof. 
\end{proof}

In summary, we have a criterion for birationality.
\begin{thm}\label{cy b main} 
Let $(X,L,T)$ be a polarized triple. Keep the same notation as in Subsection \ref{cy b keythm}. Then $|K_X+mL+T|$ gives a birational map if 
$$m> \max\left\{m_0+m_1+\rho_0-1,\mu_0+m_1+\frac{2}{\zeta}\right\}.$$
\end{thm}

This theorem is optimal in some sense by the following examples.
\begin{example}[{\cite[14.3 Theorem]{Fletcher}}]
Consider the general weighted hypersurface $X_{10}\subset \bP(1,1,1,2,5)$ which is a smooth Calabi--Yau $3$-fold. Take $L=\OO_X(1)$ and $T=K_X\sim 0$. Then $|5L|$ gives a birational map but $|4L|$ does not.

On the other hand, we may take $m_0=m_1=\mu_0=\rho_0=1$ and  $\zeta\geq 1$. Hence Theorem \ref{cy b main} implies that $|mL|$ gives a birational map for all $m\geq 5$.
\end{example}

\begin{example}[{\cite[14.3 Theorem]{Fletcher}}]
Consider the general weighted hypersurface $X_{8}\subset \bP(1,1,1,1,4)$ which is a smooth Calabi--Yau $3$-fold. Take $L=\OO_X(1)$ and $T=K_X\sim 0$. Then $|4L|$ gives a birational map but $|3L|$ does not.

On the other hand, we may take $m_0=m_1=\mu_0=\rho_0=1$.  Note that $S\in |L|$ and $C\in |L|_S|$. Hence $\zeta=L^3=2$. Then Theorem \ref{cy b main} implies that $|mL|$ gives a birational map for all $m\geq 4$.
\end{example}

\begin{example}[{\cite[Theorem 4.5]{CCC}}]
Consider the general weighted complete intersection $X_{2,6}\subset \bP(1,1,1,1,1,3)$ which is a terminal Calabi--Yau $3$-fold. Take $L=\OO_X(1)$ and $T=K_X\sim 0$. Then $|3L|$ gives a birational map but $|2L|$ does not.

On the other hand, we may take $m_0=m_1=\mu_0=\rho_0=1$.  Note that $S\in |L|$ and $C\in |L|_S|$. Hence $\zeta=L^3=4$. Then Theorem \ref{cy b main} implies that $|mL|$ gives a birational map for all $m\geq 3$.
\end{example}

\section{Birationality on polarized triples}\label{cy section 3}
In this section, we consider the birationality problem on polarized triples. By Theorem \ref{cy b main}, we need to estimate $m_0$, $m_1$, $\rho_0$, $\mu_0$, and $\zeta$. First we will give estimation for $\rho_0$ by Reid's formula. Then we reduce the estimation of $m_1$ to the estimation of Hilbert function of $L$ so that we can estimate both $m_0$ and $m_1$ by Reid's formula. Note that $\mu_0$ can be estimated by Remark \ref{cy upper mu0} and $\zeta$ can be estimated by Proposition \ref{cy b zeta} once we have estimation of $m_0$ and $m_1$.

We always assume that $(X,L,T)$ is a polarized triple with $\chi(\OO_X)>0$ in this section.

\subsection{Estimation of $\rho_0$}

In this subsection, we estimate $\rho_0$.
Note that by Theorem \ref{cy fact}(4)(5) and the fact that $i(X)|I(X)$, we have
$$
i(X)\in \{2,3,4,5,6,8,10,12\}.
$$

Since we need to estimate the Hilbert function of some divisor $D$, we need to estimate the singular part $c_Q(D)$ in Reid's formula. We list all the possible values for $c_Q(D)$ with all the possible singularities in Table A. The first row corresponds to the local index $i_Q(D)$ of $D$ and the first column corresponds to the possible singularities of $Q$. In the estimation, we will always replace $c_Q(D)$ by the minimal value in the list corresponding to $Q$.

{\small 
\smallskip

\begin{center}
\begin{tabular}{|c|c|c|c|c|c|c|}
\hline
$i_Q(D)$ & 0 & 1 & 2 & 3 & 4& 5 \\
\hline
$(1,2)$ & 0 & $-1/8$ & &  & &   \\
\hline
$(1,3)$ & 0 & $-2/9$ & $-1/9$ &  & & \\
\hline
$(1,4)$ & 0 & $-5/16$ & $-1/4$ & $-1/16$ & & \\
\hline
$(1,5)$ & 0 & $-2/5$ & $-2/5$ & $-1/5$ & 0&   \\
\hline
$(2,5)$ & 0 & $-2/5$ & $-1/5$ & $-1/5$ & $-1/5$&  \\
\hline
$(1,6)$ & 0 & $-35/72$ & $-5/9$ & $-3/8$ & $-1/9$ & $5/72$  \\
\hline
$(1,8)$ & 0 & $-21/32$ & $-7/8$ & $-25/32$ & $-1/2$& $-5/32$ \\
\hline
$(3,8)$ & 0 &$-21/32$ & $-3/8$ & $-9/32$ & $-1/2$& $-5/32$\\
\hline
$(1,10)$ & 0 & $-33/40$ & $-6/5$ & $-49/40$ & $-1$& $-5/8$  \\
\hline
$(3,10)$ & 0 & $-33/40$ & $-3/5$ & $-9/40$ & $-3/5$& $-5/8$ \\
\hline
$(1,12)$ & 0 & $-143/144$ & $-55/36$ &$-27/16$ & $-14/9$& $-175/144$ \\
\hline
$(5,12)$ & 0 &  $-143/144$ & $-19/36$ &$-11/16$ & $-5/9$& $-31/144$ \\
\hline
\hline
$i_Q(D)$ & 6& 7& 8 & 9 & 10 & 11\\
\hline
$(1,8)$ & $1/8$& $7/32$&   &  &  & \\
\hline
$(3,8)$ &$-3/8$& $-9/32$&   &  &  &\\
\hline
$(1,10)$ & $-1/5$& $7/40$& $2/5$ & $3/8$ &  & \\
\hline
$(3,10)$ & $-1/5$& $-9/40$& $-3/5$ & $-9/40$ &  &\\
\hline
$(1,12)$ &$-3/4$& $-35/144$& $2/9$ & $9/16$ & $25/36$ & $77/144$\\
\hline
$(5,12)$ &$-3/4$& $-35/144$& $-7/9$ & $-7/16$ & $-11/36$ & $-67/144$\\
\hline
\end{tabular}
\end{center}
\centerline{\underline{Table A: table of $c_Q(D)$}}}
\bigskip

Note that for a singular point $Q$ of index $r\in \{2,3,4,5\}$ and for any Weil divisor $D$,
$$
c_Q(D)\geq -\frac{r^2-1}{12r}.
$$ 

To estimate $\rho_0$, we discuss on the value of $i(X)$.
Fix a Weil divisor  $T'\equiv 0$. Recall that  $L^3\geq \frac{1}{i(X)}$ and $\lambda(L)\geq \frac{1}{i(X)}$.

If $i(X)\in \{2,3,4,5\}$, by Reid's formula and equality (\ref{cy chi}),
\begin{align*}
h^0(mL+T')%= {}&\chi(\OO_X)+\frac{m^3-m}{6}L^3+m\lambda(L)+\sum_Qc_Q(mL+T')\\
\geq{}& \chi(\OO_X)+\frac{m^3-m}{6}L^3+m\lambda(L)-\sum_Q\frac{r_Q^2-1}{12r_Q}\\
\geq{}&\frac{m^3+5m}{6i(X)}-\chi(\OO_X).
\end{align*}
Recall that $\chi(\OO_X)\leq 4$ (or $\chi(\OO_X)=1$ if $i(X)$=5),
hence 
$$
\rho_0\leq \begin{cases}
3, & \text{if }i(X)=5;\\
4, &\text{if } i(X)\in \{2,3\};\\
5, & \text{if } i(X)=4.
\end{cases}
$$

If $i(X)=6$, then we write $B_X=\{a\times(1,2),b\times(1,3),c\times(1,6)\}$. By equality (\ref{cy chi}),
$$
24\chi(\OO_X)=\frac{3}{2}a+\frac{8}{3}b+\frac{35}{6}c.
$$
Hence $c<\frac{144}{35}\chi(\OO_X)$. By Reid's formula and equality (\ref{cy chi}),
\begin{align*}
h^0(mL+T')%= {}&\chi(\OO_X)+\frac{m^3-m}{6}L^3+m\lambda(L)+\sum_Qc_Q(mL+T')\\
\geq{}& \chi(\OO_X)+\frac{m^3-m}{6}L^3+m\lambda(L)-\frac{1}{8}a-\frac{2}{9}b-\frac{5}{9}c\\
={}&\frac{m^3-m}{6}L^3+m\lambda(L)-\chi(\OO_X)-\frac{5}{72}c\\
>{}&\frac{m^3+5m}{36}-\frac{9}{7}\chi(\OO_X).
\end{align*}
Recall that $\chi(\OO_X)\leq 4$,
hence $\rho_0\leq 6$.

If $i(X)=8$,  by Morrison \cite[Proposition 3]{M=0}, we have $i(X)=I(X)$, $\chi(\OO_X)=1$, and $B_X=\{3\times(1,2),(1,4), (b_1,8),(b_2,8)\}$ for $b_1,b_2=1$ or $3$.  By Reid's formula,
\begin{align*}
h^0(mL+T')%= {}&1+\frac{m^3-m}{6}L^3+m\lambda(L)+\sum_Qc_Q(mL+T')\\
\geq{}& 1+\frac{m^3-m}{6}L^3+m\lambda(L)-3\times\frac{1}{8}-\frac{5}{16}-2\times \frac{7}{8}\\
={}&\frac{m^3+5m}{48}-\frac{23}{16}.
\end{align*}
Hence $\rho_0\leq 4$.

If $i(X)=10$, by Morrison \cite[Proposition 3]{M=0}, we have $i(X)=I(X)$, $\chi(\OO_X)=1$, and $B_X=\{3\times(1,2),(b_1,5),(b_2,5),(c,10)\}$ for $b_1,b_2=1$ or $2$, $c=1$ or $3$.  By Reid's formula,
\begin{align*}
h^0(mL+T')%= {}&1+\frac{m^3-m}{6}L^3+m\lambda(L)+\sum_Qc_Q(mL+T')\\
\geq{}& 1+\frac{m^3-m}{6}L^3+m\lambda(L)-3\times\frac{1}{8}-2\times \frac{2}{5}-\frac{49}{40}\\
={}&\frac{m^3+5m}{60}-\frac{7}{5}.
\end{align*}
Hence $\rho_0\leq 5$.

If $i(X)=12$, recall that by Morrison \cite[Proposition 3]{M=0}, we have $i(X)=I(X)$, $\chi(\OO_X)=1$, and $B_X=\{2\times(1,2),2\times(1,3),(1,4),(b, 12)\}$ for $b=1$ or $5$.  By Reid's formula,
\begin{align*}
h^0(mL+T')%= {}&1+\frac{m^3-m}{6}L^3+m\lambda(L)+\sum_Qc_Q(mL+T')\\
\geq{}& 1+\frac{m^3-m}{6}L^3+m\lambda(L)-2\times\frac{1}{8}-2\times \frac{2}{9}-\frac{5}{16}-\frac{27}{16}\\
={}&\frac{m^3+5m}{72}-\frac{61}{36}.
\end{align*}
Hence $\rho_0\leq 5$.

In summary, we proved the following proposition.
\begin{prop}\label{cy rho0}
 
We have the following estimation for $\rho_0$:
$$
\rho_0\leq \begin{cases}
3, &\text{if } i(X) =5;\\
4, &\text{if } i(X)\in \{2,3,8\};\\
5, &\text{if }  i(X)\in \{4,10,12\};\\
6, &\text{if }  i(X)=6.
\end{cases}
$$
\end{prop}

\subsection{Estimation of $m_1$}

We give a criterion for a linear system not composing with a pencil of surfaces by looking at its Hilbert function.
\begin{prop}\label{cy b non-pencil} Let $L_0$ be a nef and big Weil divisor.
If $$h^0(mL_0)>i(X)L_0^3m+1$$ for some integer $m$, then
$|mL_0|$ is not composed with a pencil of surfaces.
\end{prop}
\begin{proof} Assume that $|mL_0|$ is composed with a pencil of surfaces. Set $D:=mL_0$ and keep the same notation as in Subsection \ref{cy b setting}.
Then we have $m\pi^*(L_0)\geq M\equiv (h^0(mL_0)-1)S$. 
Note that by
Lemma \ref{cy L3}, $i(X)\pi^*(L_0)^2\cdot S$ is an integer. On the other hand, $\pi^*(L_0)^2\cdot S$ is positive since $\pi^*(L_0)|_S$ is nef and big on $S$. Hence $\pi^*(L_0)^2\cdot S\geq \frac{1}{i(X)}$.
Thus we have $mL_0^3\geq
(h^0(mL_0)-1)(\pi^*(L_0)^2\cdot S)\geq \frac{1}{i(X)}(h^0(mL_0)-1)$, a contradiction.  \end{proof}

\subsection{Proof of Theorems \ref{cy main1} and \ref{cy main2}}

In this subsection, we prove  Theorems \ref{cy main1} and \ref{cy main2} by estimating $m_0$ and $m_1$.

\begin{proof}[Proof of Theorem \ref{cy main1}]
To prove Theorem \ref{cy main1}, by Section \ref{cy Gor section}, we only need to consider polarized triples $(X,L,T)$ with $\chi(\OO_X)>0$. We discuss on the value of $i(X)$. Recall that 
$$
i(X)\in \{2,3,4,5,6,8,10,12\}.
$$
Recall again that  $L^3\geq \frac{1}{i(X)}$ and $\lambda(L)\geq \frac{1}{i(X)}$. The main problem is to estimate $\sum_Qc_Q$.
In the proof, we often use the fact that if $Q$ is a cyclic singular point and $D$ is a Weil divisor with local index $i_Q(D)=0$, then $c_Q(D)=0$.
\smallskip

{\bf Case 1.} $i(X)=2$ or  $3$.

In this case,  by Reid's formula,
\begin{align*}
h^0(i(X)L)\geq {}&\chi(\OO_X)+\frac{i(X)^3}{6}L^3>1,\\
%h^0(5L_0)\geq {}&1+\frac{5^3}{6}L_0^3+\sum_Q c_Q(5L_0)\\
%\geq {}&1+\frac{5^3}{30}\\
%>{}&5.\\
%h^0(8L_0)\geq {}&1+\frac{8^3}{6}L_0^3+\sum_Q c_Q(8L_0)\\
%\geq {}&1+\frac{8^3}{60}-2\times \frac{2}{5}\\
%>{}&8.\\
%h^0(8L_0)\geq {}&1+\frac{8^3}{6}L_0^3+\sum_Q c_Q(8L_0)\\
%\geq {}&1+\frac{8^3}{72}-2\times \frac{2}{9}\\
%>{}&7.
h^0(2i(X)L)\geq {}&\chi(\OO_X)+\frac{8i(X)^3}{6}L^3>2i(X)^2L^3+1.
\end{align*}
Hence we can take $L_0=L$, $m_0=i(X)$, and $m_1=2i(X)$. Then we have $\mu_0\leq i(X)$ by Remark \ref{cy upper mu0}. By Proposition \ref{cy b zeta}, $\zeta\geq \frac{1}{i(X)}$. By Proposition \ref{cy rho0}, $\rho_0\leq 4$.  By Theorem \ref{cy b main}, $|K_X+mL+T|$ gives a birational map for $m\geq 5i(X)+1$.
\smallskip

{\bf Case 2.} $i(X)=4$.

In this case, 
by the proof of Lemma \ref{cy lambda},
$$
\sum_{i=0}^{3}h^0(5L+iK_X)=4\lambda(5L).
$$
Hence there exists $i_0$ such that 
$$
h^0(5L+i_0K_X)\geq \lambda(5L).
$$
Take $L_0=L+i_0K_X$. Then
\begin{align*}
h^0(5L_0)= {}&h^0(5L+i_0K_X+4i_0K_X)\\
={}&h^0(5L+i_0K_X)\\
\geq {}&\lambda(5L)\\
={}&20 L_0^3+5\lambda(L)\\
> {}& 5i(X)L_0^3+1.
\end{align*}
On the other hand,
$$
h^0(4L_0)= \chi(\OO_X)+\frac{4^3-4}{6}L_0^3+4\lambda(L)>4.
$$
Hence $h^0(4L_0)\geq 5$ and $|5L_0|$ is not composed with a pencil. Take $m_0=4$. By Proposition \ref{cy rho0}, $\rho_0\leq 5$.  

If  $|4L_0|$ is composed with a pencil, then we have $\mu_0\leq 1$ by Remark \ref{cy upper mu0} and we can take $m_1=5$. By Proposition \ref{cy b zeta},   $\zeta\geq \frac{1}{2}$. By Theorem \ref{cy b main}, $|K_X+mL+T|$ gives a birational map for $m\geq 14$.

If   $|4L_0|$ is not composed with a pencil, then we have $\mu_0\leq 4$  and we can take $m_1=4$. By Proposition \ref{cy b zeta},  $\zeta\geq \frac{1}{2}$. By Theorem \ref{cy b main}, $|K_X+mL+T|$ gives a birational map for $m\geq 13$.
\smallskip

{\bf Case 3.} $i(X)=6$.

In this case, recall that $1\leq \chi(\OO_X)\leq 4$ and we write $B_X=\{a\times(1,2),b\times(1,3),c\times(1,6)\}$. By equality (\ref{cy chi}),
$$
24\chi(\OO_X)=\frac{3}{2}a+\frac{8}{3}b+\frac{35}{6}c.
$$

If $\chi(\OO_X)=1$, there is only one solution satisfying $i(X)=6$, which is  $B_X=\{5\times(1,2),4\times(1,3),(1,6)\}$. 
We can take $L_0=L+i_0K_X$ for some $i_0$ such that the local index of $L_0$ at the point $(1,6)$ is $0$. Note that 
$$
\sum_Q c_Q(kL_0)\geq \begin{cases}
-5 \times\frac{1}{8}, &\text{if } k=3;\\
-4 \times\frac{2}{9}, &\text{if } k=4;\\
-5 \times\frac{1}{8}-4 \times\frac{2}{9}, &\text{if } k=7.
%-\frac{3}{8}, &\text{if } r=6.
\end{cases}
$$
 By Reid's formula,
\begin{align*}
h^0(3L_0)= {}&1+\frac{3^3-3}{6}L_0^3+3\lambda(L)+\sum_Q c_Q(3L_0)%\\
%\geq {}&1+\frac{3^3+15}{36}-5 \times\frac{1}{8}\\
> 1,\\
h^0(4L_0)= {}&1+\frac{4^3-4}{6}L_0^3+4\lambda(L)+\sum_Q c_Q(4L_0)%\\
%\geq {}&1+\frac{4^3+20}{36}-4 \times\frac{2}{9}\\
> 2,\\
%h^0(8L_0)\geq {}&1+\frac{8^3}{6}L_0^3+\sum_Q c_Q(8L_0)\\
%\geq {}&1+\frac{8^3}{60}-2\times \frac{2}{5}\\
%>{}&8.\\
%h^0(8L_0)\geq {}&1+\frac{8^3}{6}L_0^3+\sum_Q c_Q(8L_0)\\
%\geq {}&1+\frac{8^3}{72}-2\times \frac{2}{9}\\
%>{}&7.
h^0(7L_0)= {}&1+\frac{7^3-7}{6}L_0^3+7\lambda(L)+\sum_Q c_Q(7L_0)%\\
%\geq {}&1+\frac{7^3}{6}L_0^3-5 \times\frac{1}{8}-4 \times\frac{2}{9}\\
> 7i(X)L_0^3+1.
\end{align*}
Hence $h^0(3L_0)\geq 2$ and $|7L_0|$ is not composed with a pencil. Take $m_0=3$. By Proposition \ref{cy rho0}, $\rho_0\leq 6$.

If $|4L_0|$ and $|3L_0|$ are composed with the same pencil, then we have $\mu_0\leq 2$ by Remark \ref{cy upper mu0}. Take $m_1=7$. By Proposition \ref{cy b zeta},  $\zeta\geq \frac{1}{3}$. By Theorem \ref{cy b main}, $|K_X+mL+T|$ gives a birational map for $m\geq 16$.

If $|4L_0|$ and $|3L_0|$ are not composed with the same pencil, then we can take $m_1=4$ and we have $\mu_0\leq 3$. By Proposition \ref{cy b zeta},  $\zeta\geq \frac{1}{2}$. By Theorem \ref{cy b main}, $|K_X+mL+T|$ gives a birational map for $m\geq 13$.

Now we assume that $\chi(\OO_X)\geq 2$. 
Note that for any Weil divisor $D$ and singular point  $Q$ of index $r$, 
$$
c_Q(3D)+c_Q(3D+3K_X)=\begin{cases}
-\frac{1}{8}, &\text{if } r=2;\\
0, &\text{if } r=3;\\
-\frac{3}{8}, &\text{if } r=6.
\end{cases}
$$
Hence 
\begin{align*}
{}& h^0(3L)+h^0(3L+3K_X)\\
={}&2\chi(\OO_X)+2\lambda(3L)+\sum_Q(c_Q(3L)+c_Q(3L+3K_X))\\
={}&2\chi(\OO_X)+2\lambda(3L)-\frac{1}{8}a-\frac{3}{8}c\\
\geq {}&2\lambda(3L).
\end{align*}
Therefore there exists a Weil divisor $L_0\equiv L$ such that 
\begin{align*}
 h^0(3L_0)  \geq \lambda(3L)
 = 4L^3+3\lambda(L)
 >1.
\end{align*}
On the other hand,
$$
h^0(6L_0)\geq \chi(\OO_X)+\frac{6^3}{6}L_0^3 >6i(X)L_0^3+1.
$$
Hence $|6L_0|$ is not composed with a pencil.

Hence we can take $m_0=3$ and $m_1=6$. Then we have $\mu_0\leq 3$ by Remark \ref{cy upper mu0}. By Proposition \ref{cy b zeta}, $\zeta\geq \frac{1}{3}$. By Proposition \ref{cy rho0}, $\rho_0\leq 6$.  By Theorem \ref{cy b main}, $|K_X+mL+T|$ gives a birational map for $m\geq 16$.

\smallskip

{\bf Case 4.} $i(X)=5$.

In this case, recall that by Morrison \cite[Proposition 3]{M=0}, we have $i(X)=I(X)$, $\chi(\OO_X)=1$, and $B_X=\{(b_1,5),(b_2,5),(b_3,5),(b_4,5),(b_5,5)\}$ for $b_i=1$ or $2$ for $1\leq i\leq 5$. 
We can take $L_0=L+i_0K_X$ for some $i_0$ such that the local index of $L_0$ at the point $(b_1,5)$ is $0$. Note that
$\sum_Q c_Q(kL_0) 
\geq -4 \times\frac{2}{5}$ for all $k$.
By Reid's formula,
\begin{align*}
h^0(4L_0)= {}&1+\frac{4^3-4}{6}L_0^3+4\lambda(L)+\sum_Q c_Q(4L_0)%\\
%\geq {}&1+\frac{4^3+20}{30}-4 \times\frac{2}{5}\\
> 2,\\
h^0(5L_0)= {}&1+\frac{5^3-5}{6}L_0^3+5\lambda(L)%\\
%\geq {}&1+\frac{5^3+25}{30}\\
= 6,\\
%h^0(8L_0)\geq {}&1+\frac{8^3}{6}L_0^3+\sum_Q c_Q(8L_0)\\
%\geq {}&1+\frac{8^3}{60}-2\times \frac{2}{5}\\
%>{}&8.\\
%h^0(8L_0)\geq {}&1+\frac{8^3}{6}L_0^3+\sum_Q c_Q(8L_0)\\
%\geq {}&1+\frac{8^3}{72}-2\times \frac{2}{9}\\
%>{}&7.
h^0(6L_0)= {}&1+\frac{6^3-6}{6}L_0^3+6\lambda(L)+\sum_Q c_Q(6L_0)%\\
%\geq {}&1+35L_0^3+6\lambda(L)-4 \times\frac{2}{5}\\
> 6i(X)L_0^3+1.
\end{align*}
Hence $h^0(4L_0)\geq 3$ and $|6L_0|$ is not composed with a pencil. Take $m_0=4$. By Proposition \ref{cy rho0}, $\rho_0\leq 3$.

If $|5L_0|$ and $|4L_0|$ are composed with the same pencil, then we have $\mu_0\leq 1$ by Remark \ref{cy upper mu0}. Take $m_1=6$. By Proposition \ref{cy b zeta},  $\zeta\geq \frac{2}{5}$. By Theorem \ref{cy b main}, $|K_X+mL+T|$ gives a birational map for $m\geq 13$.

If $|4L_0|$ is composed with a pencil, and $|5L_0|$ and $|4L_0|$ are not composed with the same pencil, then we can take $m_1=5$ and we have $\mu_0\leq 2$. By Proposition \ref{cy b zeta},  $\zeta\geq \frac{2}{5}$. By Theorem \ref{cy b main}, $|K_X+mL+T|$ gives a birational map for $m\geq 13$.

If $|4L_0|$ is not composed with a pencil,  then we can take $m_1=4$ and we have $\mu_0\leq 4$. By Proposition \ref{cy b zeta},  $\zeta\geq \frac{2}{5}$. By Theorem \ref{cy b main}, $|K_X+mL+T|$ gives a birational map for $m\geq 14$.
\smallskip

{\bf Case 5.} $i(X)=8$.

In this case, recall that by Morrison \cite[Proposition 3]{M=0}, we have $i(X)=I(X)$, $\chi(\OO_X)=1$, and $B_X=\{3\times(1,2),(1,4), (b_1,8),(b_2,8)\}$ for $b_1,b_2=1$ or $3$. 
We can take $L_0=L+i_0K_X$ for some $i_0$ such that the local index of $L_0$ at the point $(b_1,8)$ is $0$. Note that
$$
\sum_Q c_Q(kL_0)\geq \begin{cases}
- \frac{7}{8}, &\text{if } k=4;\\
-\frac{5}{16}- \frac{7}{8}, &\text{if } k=6.
%-\frac{3}{8}, &\text{if } r=6.
\end{cases}
$$
By Reid's formula,
\begin{align*}
h^0(4L_0)= {}&1+\frac{4^3-4}{6}L_0^3+4\lambda(L)+\sum_Q c_Q(4L_0)
%\geq {}&1+\frac{4^3}{6}L_0^3+\sum_Q c_Q(4L_0)\\
%\geq {}&1+\frac{4^3}{48}- \frac{7}{8}\\
> 1,\\
h^0(6L_0)= {}&1+\frac{6^3-6}{6}L_0^3+6\lambda(L)+\sum_Q c_Q(6L_0)
%\geq {}&1+\frac{6^3}{6}L_0^3+\sum_Q c_Q(6L_0)\\
%\geq {}&1+\frac{6^3}{48}-\frac{5}{16}- \frac{7}{8}\\
> 4,\\
%h^0(8L_0)\geq {}&1+\frac{8^3}{6}L_0^3+\sum_Q c_Q(8L_0)\\
%\geq {}&1+\frac{8^3}{60}-2\times \frac{2}{5}\\
%>{}&8.\\
%h^0(8L_0)\geq {}&1+\frac{8^3}{6}L_0^3+\sum_Q c_Q(8L_0)\\
%\geq {}&1+\frac{8^3}{72}-2\times \frac{2}{9}\\
%>{}&7.
h^0(8L_0)= {}&1+\frac{8^3-8}{6}L_0^3+8\lambda(L) 
%\geq {}&1+\frac{8^3}{6}L_0^3+\sum_Q c_Q(8L_0)\\
%= {}&1+\frac{8^3}{6}L_0^3\\
> 8i(X)L_0^3+1.
\end{align*}
Hence $h^0(4L_0)\geq 2$ and $|8L_0|$ is not composed with a pencil. Take $m_0=4$. By Proposition \ref{cy rho0}, $\rho_0\leq 4$.

If $|6L_0|$ and $|4L_0|$ are composed with the same pencil, then we have $\mu_0\leq \frac{6}{4}$ by Remark \ref{cy upper mu0}. Take $m_1=8$. By Proposition \ref{cy b zeta},  $\zeta\geq \frac{3}{8}$. By Theorem \ref{cy b main}, $|K_X+mL+T|$ gives a birational map for $m\geq 16$.

If $|6L_0|$ and $|4L_0|$ are not composed with the same pencil, then we can  take $m_1=6$ and we have $\mu_0\leq 4$. By Proposition \ref{cy b zeta}, $\zeta\geq \frac{3}{8}$. By Theorem \ref{cy b main}, $|K_X+mL+T|$ gives a birational map for $m\geq 16$.
\smallskip

{\bf Case 6.} $i(X)=10$.

In this case, recall that by Morrison \cite[Proposition 3]{M=0}, we have $i(X)=I(X)$, $\chi(\OO_X)=1$, and $B_X=\{3\times(1,2),(b_1,5),(b_2,5),(c,10)\}$ for $b_1,b_2=1$ or $2$, $c=1$ or $3$. 
We can take $L_0=L+i_0K_X$ for some $i_0$ such that the local index of $L_0$ at the point $(c,10)$ is $0$. Note that for an even integer $k$, $\sum_Q c_Q(kL_0)\geq  -2\times \frac{2}{5}$.
By Reid's formula,
\begin{align*}
%h^0(2L_0)\geq {}&1+\frac{2^3}{6}L_0^3+\sum_Q c_Q(2L_0)\\
%\geq {}&1+\frac{2^3}{60}-2\times \frac{2}{5}\\
%>{}&0.\\
h^0(4L_0)={}&1+\frac{4^3-4}{6}L_0^3+4\lambda(L)+\sum_Q c_Q(4L_0)>1;\\
%\geq {}&1+\frac{4^3+20}{60}-2\times \frac{2}{5}\\
%>&1,\\
h^0(6L_0)= {}&1+\frac{6^3-6}{6}L_0^3+6\lambda(L)+\sum_Q c_Q(6L_0)>4;\\
%\geq {}&1+\frac{6^3+30}{60}-2\times \frac{2}{5}\\
%>&4,\\
%h^0(8L_0)\geq {}&1+\frac{8^3}{6}L_0^3+\sum_Q c_Q(8L_0)\\
%\geq {}&1+\frac{8^3}{60}-2\times \frac{2}{5}\\
%>{}&8.\\
%h^0(8L_0)\geq {}&1+\frac{8^3}{6}L_0^3+\sum_Q c_Q(8L_0)\\
%\geq {}&1+\frac{8^3}{72}-2\times \frac{2}{9}\\
%>{}&7.
h^0(8L_0)= {}&1+\frac{8^3-8}{6}L_0^3+8\lambda(L)+\sum_Q c_Q(8L_0)%\\
%\geq {}&1+84L_0^3+8\lambda(L)-2\times \frac{2}{5}
>8i(X)L_0^3+1.
\end{align*}
Hence $h^0(4L_0)\geq 2$ and $|8L_0|$ is not composed with a pencil. Take $m_0=4$. By Proposition \ref{cy rho0}, $\rho_0\leq 5$.

If $|6L_0|$ and $|4L_0|$ are composed with the same pencil, then we have $\mu_0\leq \frac{3}{2}$ by Remark \ref{cy upper mu0}. Take $m_1=8$. By Proposition \ref{cy b zeta}, $\zeta\geq \frac{3}{10}$. By Theorem \ref{cy b main}, $|K_X+mL+T|$ gives a birational map for $m\geq 17$.

If $|6L_0|$ and $|4L_0|$ are not composed with the same pencil, then we can take  $m_1=6$ and we have $\mu_0\leq 4$. By Proposition \ref{cy b zeta}, $\zeta\geq \frac{3}{10}$. By Theorem \ref{cy b main}, $|K_X+mL+T|$ gives a birational map for $m\geq 17$.
\smallskip

{\bf Case 7.} $i(X)=12$.

In this case, recall that by Morrison \cite[Proposition 3]{M=0}, we have $i(X)=I(X)$, $\chi(\OO_X)=1$, and $B_X=\{2\times(1,2),2\times(1,3),(1,4),(b, 12)\}$ for $b=1$ or $5$. 
We can take $L_0=L+i_0K_X$ for some $i_0$ such that the local index of $L_0$ at the point $(b,12)$ is $0$.  Note that
$$
\sum_Q c_Q(kL_0)\geq \begin{cases}
-2\times \frac{1}{8}-\frac{5}{16}, &\text{if } k=3,9;\\
-\frac{5}{16}, &\text{if } k=6.
%-\frac{3}{8}, &\text{if } r=6.
\end{cases}
$$
By Reid's formula,
\begin{align*}
%h^0(L_0)\geq {}&1+\frac{1}{6}L_0^3+\sum_Q c_Q(L_0)\\
%\geq {}&1+\frac{1}{72}-2\times \frac{1}{8}-2\times \frac{2}{9}-\frac{5}{16}\\
%>{}&0.\\
h^0(3L_0)= {}&1+\frac{3^3-3}{6}L_0^3+3\lambda(L)+\sum_Q c_Q(3L_0)%\\
%\geq {}&1+\frac{3^3+15}{72}-2\times \frac{1}{8}-\frac{5}{16}\\
> 1,\\
h^0(6L_0)= {}&1+\frac{6^3-6}{6}L_0^3+6\lambda(L)+\sum_Q c_Q(6L_0)%\\
%\geq {}&1+\frac{6^3+30}{72}-\frac{5}{16}\\
> 4,\\
%h^0(7L_0)\geq {}&1+\frac{7^3}{6}L_0^3+\sum_Q c_Q(7L_0)\\
%\geq {}&1+\frac{7^3}{72}-2\times \frac{1}{8}-2\times \frac{2}{9}-\frac{5}{16}\\
%>{}&4,\\
%h^0(8L_0)\geq {}&1+\frac{8^3}{6}L_0^3+\sum_Q c_Q(8L_0)\\
%\geq {}&1+\frac{8^3}{72}-2\times \frac{2}{9}\\
%>{}&7.
h^0(9L_0)
= {}&1+\frac{9^3-9}{6}L_0^3+9\lambda(L)+\sum_Q c_Q(9L_0)
%\geq {}&1+\frac{9^3}{6}L_0^3+\sum_Q c_Q(9L_0)%\\
%\geq {}&1+\frac{9^3}{6}L_0^3-2\times \frac{1}{8}-\frac{5}{16}\\
> 9i(X)L_0^3+1.
\end{align*}
Hence $h^0(3L_0)\geq 2$ and $|9L_0|$ is not composed with a pencil. Take $m_0=3$. By Proposition \ref{cy rho0}, $\rho_0\leq 5$.

If $|6L_0|$ and $|3L_0|$ are composed with the same pencil, then we have $\mu_0\leq \frac{3}{2}$ by Remark \ref{cy upper mu0}. Take $m_1=9$. By Proposition \ref{cy b zeta},  $\zeta\geq \frac{1}{3}$. By Theorem \ref{cy b main}, $|K_X+mL+T|$ gives a birational map for $m\geq 17$.

If $|6L_0|$ and $|3L_0|$ are not composed with the same pencil, then we have $\mu_0\leq 3$  and we can take $m_1=6$. By Proposition \ref{cy b zeta},  $\zeta\geq \frac{1}{3}$. By Theorem \ref{cy b main}, $|K_X+mL+T|$ gives a birational map for $m\geq 16$. 
\end{proof}

\begin{proof}[Proof of Theorem \ref{cy main2}]
Since $L$ has no stable components, take a sufficient divisible $k$ such that $kL\sim M$ is movable and effective and take a sufficient small rational number $\delta>0$ such that $(X, \delta M)$ is terminal. Run a $(K_X+\delta M)$-MMP with scaling of an ample divisor, it terminates on $X'$ by Kawamata \cite{K3}. Since $i(X)l(K_X+\delta M)\sim i(X)l\delta M$ is movable for $l$ sufficient divisible, this MMP $\psi:X\dashrightarrow X'$ does not contract any divisors, hence isomorphic in codimension one. Hence $(X', \delta \psi_*M)$ is terminal and so is $X'$. Hence $X'$ is a minimal $3$-fold with $K_{X'}=\psi_*K_X\equiv 0$ and $\psi_*L$ is a nef and big Weil divisor by MMP. Note that $|K_X+mL+T|$ gives a birational map if and only if so does $|K_{X'}+m\psi_*L+\psi_*T|$,  hence Theorem \ref{cy main2} follows from Theorem \ref{cy main1}. 
\end{proof}


\begin{thebibliography}{99}
 
\bibitem{Ando} Ando, T.:  Pluricanonical systems of algebraic
varieties of general type of dimension $\leq 5$. In: Algebraic geometry (Sendai, 1985), Adv. Stud. Pure Math.,  10, pp. 1--10. North-Holland, Amsterdam (1987)

%\bibitem{BCHM} C. Birkar, P. Cascini, C. D. Hacon and J. M$^c$Kernan, \emph{Existence
%of minimal models for varieties of log general type}, J. Amer. Math.
%Soc. 23 (2010), 405-468.

 \bibitem{Bomb}  Bombieri, E.: Canonical models of surfaces of general type. Publ. Math. IHES 42, 171--219 (1973)

\bibitem{CCC}  Chen, J-J.,  Chen, J.A., Chen,  M.:  On quasismooth weighted complete intersections. J. Algebraic Geom. 20, 239--262  (2011)



\bibitem{CC1} Chen, J.A., Chen,  M.:   Explicit birational geometry of threefolds of general type, I.  Ann. Sci. \'Ec. Norm. Sup\'er (4)   43, 365--394 (2010)

\bibitem{CC2} Chen, J.A., Chen,  M.:  Explicit birational geometry of threefolds of general type, II. J. Differential Geom. 86, 237--271 (2010)


\bibitem{CC3} Chen, J.A., Chen,  M.:  Explicit birational geometry of threefolds of general type, III. Compositio Math., 151, 1041--1082 (2015)



\bibitem{Chen6} Chen,  M.:   On pluricanonical maps for threefolds of general type. J. Math. Soc. Japan, 50, 615--621 (1998)

\bibitem{Camb} Chen, M.: Canonical stability in terms of
singularity index for algebraic threefolds. Math. Proc. Cambridge Philos. Soc. 13, 241--264 (2001)
 
 
 
\bibitem{C} Chen,  M.:  On anti-pluricanonical systems of $\bQ$-Fano 3-folds. Sci. China Math.  54,  1547--1560 (2011)



\bibitem{CJ} Chen, M.,  Jiang, C.:   On the anti-canonical geometry of $\bQ$-Fano threefolds. Preprint, arXiv:1408.6349, to appear in J. Differential Geom.


\bibitem{Fletcher} Iano-Fletcher, A.R.:  Working with weighted complete intersections. In:
Explicit birational geometry of 3-folds,
London Math. Soc. Lecture Note Ser., 281, pp. 101--173.
 Cambridge Univ. Press, Cambridge (2000)

\bibitem{F}   Fukuda, S.:  A note on Ando's paper ``Pluricanonical systems of algebraic varieties of general type of dimension $\leq 5$". Tokyo J. of Math.  14, 479--487 (1991)



\bibitem{K1} Kawamata, Y.:   Minimal models and the Kodaira dimension of algebraic fiber spaces. J. Reine 
Angew. Math.  363, 1--46 (1985)  

\bibitem{Ka=0} Kawamata, Y.:  On the plurigenera of minimal algebraic 3-folds with $K\equiv 0$. Math. Ann.  275,  539--546 (1986)

 \bibitem{K2} Kawamata, Y.: Crepant blowing-up of 3-dimensional canonical singularities and
its application to degeneration of surfaces. Ann. of Math. 127, 93--163 (1988)

 \bibitem{K3} Kawamata, Y.: Termination of log-flips for algebraic $3$-folds. Internat. J. Math. 3,  653--659 (1992)



\bibitem{KMM} Kawamata, Y., Matsuda, K., Matsuki, K.: Introduction
to the minimal model problem. In: Algebraic geometry (Sendai, 1985), Adv. Stud. Pure Math.,  10, pp. 283--360. North-Holland, Amsterdam (1987)  


%\bibitem{Kollar} J. Koll\'ar, {\em Shafarevich maps and automorphic forms}, M. B. Porter Lectures,
%Princeton University Press, Princeton, NJ, 1995.


\bibitem{KM} Koll\'{a}r, J., Mori, S.: Birational geometry of algebraic varieties,
Cambridge tracts in mathematics, vol. 134. Cambridge University
Press (1998)

\bibitem{Matsuki} Matsuki, K.:  On the value $n$ which makes the $n$-ple canonical map birational for a $3$-fold of general type. J. Math. Soc. Japan  38, 339--359 (1986)

\bibitem{Mi}  Miyaoka, Y.:  The Chern classes and Kodaira dimension of a minimal variety. In: Algebraic geometry (Sendai, 1985), Adv. Stud. Pure Math.,  10, pp. 449--476. North-Holland, Amsterdam (1987)   


\bibitem{M=0}  Morrison, D.:  A remark on Kawamata's paper ``On the plurigenera 
of minimal algebraic 3-folds with $K\equiv 0$". Math. Ann.  275,  547--553 (1986)

\bibitem{O}  Oguiso, K.: On polarized Calabi--Yau 3-folds. J. Fac. Sci. Univ. Tokyo 38, 395--429 (1991)
\bibitem{OP}  Oguiso, K.,  Peternell, T.:   On polarized canonical Calabi--Yau threefolds. Math. Ann.  301,  237--248 (1995)



\bibitem{YPG} Reid, M.: Young person's guide to canonical
singularities. In: Algebraic geometry, Bowdoin, 1985 (Brunswick, Maine, 1985), Proc. Sympos. Pure Math., 46, Part 1, pp. 345--414. Amer. Math. Soc., Providence, RI (1987)


\bibitem{Reider} Reider, I.: Vector bundles of rank 2 and linear systems on
algebraic surfaces. Ann. of Math.  127, 309--316 (1988)




\bibitem{T} Tankeev, S.G.:   On n-dimensional canonically
polarized varieties and varieties of fundamental type. Izv. A. N.
SSSR, Ser. Math.  35, 31--44 (1971)
\end{thebibliography}


\end{document}